\documentclass[11pt,fleqn,a4paper]{article}

\pdfoutput=1

\usepackage{pdfsync}
\usepackage{graphicx}
\usepackage{latexsym}
\usepackage[small,bf]{caption}
\usepackage{amsfonts}
\usepackage{amssymb}
\usepackage{amsmath}
\usepackage{stmaryrd}
\usepackage{bbm}
\usepackage{times}

\usepackage{hyperref}
\usepackage{setspace}
\usepackage{enumitem}
\usepackage{authblk}

\usepackage[top=3cm, bottom=3cm, outer=3cm, inner=3cm, heightrounded, marginparwidth=1.9cm, marginparsep=0.3cm]{geometry}
\usepackage{amsthm}
\makeatother
\newtheorem{thm}{Theorem}
\newtheorem{lem}{Lemma}
\newtheorem{prop}{Proposition}
\newtheorem{cor}{Corollary}
\theoremstyle{definition}

\newcommand{\N}{\mathbb{N}}
\newcommand{\R}{\mathbb{R}}
\newcommand{\E}{\mathbb{E}}
\newcommand{\Q}{\mathbb{Q}}
\renewcommand{\P}{\mathbb{P}}
\newcommand{\sgn}{\mathrm{sgn}}
\newcommand{\Lcal}{\mathcal{L}}

\begin{document}

\title{Derivation of mean-field equations for stochastic particle systems}

\author[1]{Stefan Grosskinsky}
\author[2]{Watthanan Jatuviriyapornchai}
\affil[1]{Mathematics Institute, University of Warwick, Coventry CV4 7AL, UK}
\affil[2]{Department of Mathematics, Faculty of Science, Mahidol University, Thailand}
\affil[2]{Centre of Excellence in Mathematics, Commission on higher Education, Thailand}

\maketitle

\begin{abstract}
We study stochastic particle systems on a complete graph and derive effective mean-field rate equations in the limit of diverging system size, which are also known from cluster aggregation models. 
We establish the propagation of chaos under generic growth conditions on particle jump rates, and the limit provides a master equation for the single site dynamics of the particle system, which is a non-linear birth death chain. Conservation of mass in the particle system leads to conservation of the first moment for the limit dynamics, and to non-uniqueness of stationary
distributions. Our findings are consistent with recent results on exchange driven growth, and provide a connection between the well studied phenomena of gelation and condensation.

\end{abstract}


\section{Introduction}

In the physics literature, stochastic particle systems in a limit of large system size are often described by a mean-field master equation for the time evolution of a single lattice site \cite{godreche2003dynamics,godreche2005dynamics,evans2014condensation}. For conservative systems, these equations are very similar to mean-field rate or kinetic equations in the study of cluster growth models (see e.g.\ \cite{krapivsky2010kinetic} and references therein). We focus on particle systems where only one particle jumps at a time, which corresponds to monomer exchange in cluster growth models as studied in \cite{ben2003exchange}, and also in the well-known Becker-D\"oring model \cite{becker1935kinetische,ball1986becker}. 
While these mean-field equations often provide the starting point for the analysis of stochastic particle systems in the physics literature (see e.g.\ \cite{godreche2003dynamics,godreche2005dynamics}) and have an intuitive form, to our knowledge their connection to underlying particle systems has not been rigorously studied so far. 

In this paper, we establish a law of large numbers for empirical measures and provide a rigorous derivation of this equation for misanthrope-type processes \cite{cocozza1985processus} with jump rates bounded by a bilinear function and a general class of initial conditions on a complete graph. The limit equation describes the dynamics of the fraction $f_k (t)\in [0,1]$ of lattice sites with a given occupation number $k$, and also provides the master equation of a birth death chain for the limiting single site dynamics of the process under additional symmetry assumptions on the initial condition. In this case our main result further implies the propagation of chaos and asymptotic independence of single site processes (see e.g.\ \cite{daipra2017}). Note that no time rescaling is required and the limiting dynamics are non-linear, i.e. the birth and death rates of the chain depend on the distribution $f_k (t)$ as a result of the mean-field interaction of the particle system on the complete graph. Even though the limiting birth death dynamics is irreducible under non-degenerate initial conditions, the non-linearity leads to conservation of the first moment of the initial distribution, resulting in a continuous family of stationary distributions, as has been observed before for other non-linear birth death chains (see e.g.\ \cite{karlin1957differential}). 
Existence of limits follows from standard tightness arguments, and the deterministic limit equation arises from a vanishing martingale part for the empirical processes. Previous results along these lines in the context of fluid limits include stochastic hybrid systems \cite{pakdaman2010fluid}, interacting diffusions \cite{daipra2017}, and a particular zero-range process using a different technique \cite{graham2009rate}. Our proof also includes a simple uniqueness argument for solutions of the limit equation similar to \cite{ball1986becker} and recent work \cite{esenturk2017}. 

Under certain conditions on the jump rates, stochastic particle systems
can exhibit a condensation transition where a non-zero fraction of
all particles accumulates in a condensate, provided the particle density
exceeds a critical value $\rho_c$. Condensing models with homogeneous stationary product measures have attracted significant research interest (see e.g. \cite{chleboun2014condensation,evans2014condensation}
for recent summaries), including zero-range processes of the type
introduced in \cite{drouffe1998simple,evans2000phase}, inclusion
processes with a rescaled system parameter \cite{grosskinsky2013dynamics,cao2014dynamics}
and explosive condensation models \cite{waclaw2012explosive,chau2015explosive}. While the stationary measures have been understood in great detail on a rigorous level \cite{chleboun2014condensation,jeon2000size,grosskinsky2003condensation,armendariz2009thermodynamic,armendariz2013zero},
the dynamics of these processes continue to pose interesting mathematical questions. First recent results for zero-range and inclusion processes have been obtained on metastability in the stationary dynamics of the condensate location \cite{beltran2012metastability,armendariz2015metastability,bianchi2016metastability}, approach to stationarity on fixed lattices under diverging particle density \cite{grosskinsky2013dynamics,beltran2015martingale}, and a hydrodynamic limit for density profiles below the critical value \cite{stamatakis2015hydrodynamic}. 

Our result provides a contribution towards a rigorous understanding of the approach to stationarity in the thermodynamic limit of diverging system size and particle number. This exhibits an interesting coarsening regime characterized by a power-law time evolution of typical observables, which has been identified in previous heuristic results \cite{godreche2003dynamics,grosskinsky2003condensation,chau2015explosive,godreche2016coarsening} also on finite dimensional regular lattices. Condensation implies that stationary measures for the limiting birth death dynamics only exist up to a maximal first moment $\rho_c <\infty$, above which $f_k (t)$ phase separates over time into two parts describing the mass distribution in the condensate and the background of the underlying particle system. Explicit traveling wave scaling solutions for the condensed part of the distribution have been found on a heuristic level in \cite{godreche2003dynamics,jatuviriyapornchai2016coarsening,godreche2016coarsening} for zero-range processes and in \cite{ben2003exchange} for a specific explosive condensation model.

The paper is organized as follows. In Section \ref{sec:notation} we introduce notation and state our main result with the proof given in Section \ref{sec:proof}. In Section \ref{sec:properties} we discuss basic properties of the limit dynamics and its solutions, as well as limitations and possible extensions of our result. We present particular examples of condensing systems in Section \ref{sec:examples} and provide a concluding discussion in Section \ref{sec:discussion}.

\section{Notation and main result \label{sec:notation}}
We consider a stochastic particle system $({\eta}(t):t>0)$ of misanthrope type \cite{cocozza1985processus} on finite lattices $\Lambda$ of size $|\Lambda|=L$. 
Configurations are denoted by ${\eta} =(\eta_x :x\in\Lambda)$ where $\eta_x \in\mathbb{N}_0$ is the number of particles on site $x$, and the state space is denoted by $\Omega =\mathbb{N}_0^{\Lambda}$. 
The dynamics of the process is defined by the infinitesimal generator 
\begin{equation}
\label{eq:GenMis}
(\mathcal{L}g)({\eta})=\sum_{x,y\in\Lambda}q(x,y)c(\eta_{x},\eta_{y})(g({\eta}^{x\rightarrow y})-g({\eta})) .
\end{equation}
Here the usual notation ${\eta}^{x\rightarrow y}$ indicates a configuration where one particle has moved from site $x$ to $y$, i.e. $\eta_{z}^{x\rightarrow y}=\eta_{z}-\delta_{z,x} +\delta_{z,y}$, and $\delta$ is the Kronecker delta. 
Since we focus on finite lattices only, the generator \eqref{eq:GenMis} is defined for all bounded, continuous test functions $g\in C^{b}(\Omega)$. For a general discussion and the construction of the dynamics on infinite lattices see \cite{andjel1982,balazs2007existence}. 

To ensure that the process is non-degenerate, the jump rates satisfy
\begin{equation}\label{cassum}
\left\{ \begin{array}{cl}
c(0,l)=0\;&\mbox{for all }l\geq 0\\
c(k,l)>0\;&\mbox{for all }k>0\;\mbox{and }l\geq 0.
\end{array} \right.
\end{equation}
Our main further assumption on the dynamics is that the rates grow sublinearly, in the sense that they are bounded by a bilinear function
\begin{equation}
\label{eq:lip}
c(k,l)\leq C_1 k (l+ C_2)\quad\mbox{for constants }C_1 ,C_2 >0\ .
\end{equation}
We focus on complete graph dynamics, i.e. $q(x,y)=1/(L-1)$ for all $x \neq y$, and denote by $\mathbb{P}^L$ and $\mathbb{E}^L$ the law and expectation on the path space $D_{[0,\infty)}(\Omega)$ of the process. 
As usual, we use the Borel $\sigma$-algebra for the discrete product topology on $\Omega$, and the smallest $\sigma$-algebra on $D_{[0,\infty)}(\Omega)$ such that $\omega\mapsto\eta_t(\omega)$ is measurable for all $t\geq 0$. We will study the empirical processes $t\mapsto F_k^L ({\eta}(t))$ defined by the test functions
\begin{equation}
F_{k}^L ({\eta}):=\frac{1}{L}\sum_{x\in\Lambda}\delta_{\eta_{x},k} \in [0,1], \label{fk}
\end{equation}
counting the fraction of lattice sites for each occupation number $k\geq 0$. 

In the following, we consider a sequence (in $L$) of initial conditions $\eta (0)=(\eta_x(0):x\in\Lambda)$. We first require the minimal condition that there exists a fixed probability distribution $f(0)$ on $\N_0$ with finite moments
\begin{equation}
\label{initialcon0}
\rho :=m_1 (0):=\sum_{k\geq 1} k f_k (0)<\infty\quad\mbox{and}\quad m_2 (0):=\sum_{k\geq 1} k^2 f_k (0)<\infty\ ,
\end{equation}
such that we have a weak law of large numbers
\begin{equation}
\label{initialcon0b}
F^L_{k}({\eta} (0)) \to f_k (0)\quad\mbox{in distribution for all }k\geq 0\ .
\end{equation}

We need further regularity assumptions on the initial conditions, namely a uniform bound of first and second moments,
\begin{align}
\label{initialcon0c}
\eta (0)\in\Omega_\alpha :=\Big\{\eta :\frac{1}{L}\sum_{x \in \Lambda}\eta_x \leq \alpha_1 ,\ \frac{1}{L}\sum_{x \in \Lambda}\eta_x^2\leq \alpha_2\Big\}\quad\mbox{for all }L\geq 1\ ,
\end{align}
for some fixed $\alpha_1 ,\alpha_2 >0$. This uniform condition could be replaced by a tail condition on the law of $\eta (0)$, but we impose it here to avoid uninteresting complications in the proof. We discuss this issue and specific examples of initial conditions in Section \ref{sec:ergo}. Simple choices that fulfill all conditions are for example product measures with a finite maximal occupation number per site.
Note that (\ref{initialcon0c}) and conservation of mass of the dynamics implies that
\begin{equation}
\label{alpha}
\frac{1}{L}\sum_{x\in\Lambda} \eta_x (t)=\sum_{k\geq 0}k F_k^L (\eta (t))\leq \alpha_1 \ , \quad\P^L -a.s.\ \mbox{for all }t\geq 0\mbox{ and }L\geq 1\ .
\end{equation}

Our main result is then a weak law of large numbers for the empirical processes $t\mapsto F_k^L (\eta (t))$ which holds pointwise in $k$ or, equivalently, in a weak sense, where we use the notation 
\begin{equation}
\langle F^L (\eta ),h\rangle =\sum_{k\geq 0} h_k F^L_k (\eta )\ ,
\end{equation}
for all bounded functions $h:\N_0\to\R$. 

\begin{thm}
\label{thmfactorization}
Consider a process with generator \eqref{eq:GenMis} on the complete graph with sublinear rates (\ref{eq:lip}) and initial conditions satisfying \eqref{initialcon0}, \eqref{initialcon0b} and \eqref{initialcon0c}. 
Then we have a weak law of large numbers, i.e.\ for all bounded $h:\N_0\to\R$, \begin{equation}
\big(\langle F^L (\eta (t)),h\rangle :t\geq 0\big)\to \big(\langle f(t),h\rangle :t\geq 0\big)\quad\mbox{weakly on path space as }L\to\infty\ ,
\end{equation}
where 
$t\mapsto f(t)=(f_k(t):k\in \mathbb{N}_0)$ is the unique global solution of the \textbf{mean-field equation}
\begin{align}
\label{eq:mis_diff_f_k}
\frac{df_{k}(t)}{dt}
&=\sum_{l\geq 0}c(k+1,l)f_l(t)f_{k+1}(t)+\sum_{l\geq 0} c(l,k-1) f_l(t)f_{k-1}(t) \nonumber \\
&\quad-\bigg(\sum_{l\geq 0}c(k,l)f_l(t)+\sum_{l\geq 0} c(l,k)f_l(t)\bigg)f_{k}(t)\quad\mbox{for all }k\geq 0,
\end{align}
with initial condition $f(0)$ given by (\ref{initialcon0b}). 
Here we use the convention $f_{-1}(t)\equiv 0$ for all $t\geq 0$ and recall that $c(0,l)=0$ for all $l\geq 0$.
\end{thm}
Note that this result implies in particular existence and uniqueness of the solution to \eqref{eq:mis_diff_f_k} for all $t\geq 0$. 
Let us denote expectations by
\begin{equation}
f_{k}^L(t):=\mathbb{E}^L \big [F_{k}^L ({\eta} (t))\big]=\frac{1}{L}\sum_{x \in \Lambda}\mathbb{P}^L[\eta_x(t)=k]\in [0,1],
\label{eq:f^L}
\end{equation}
and write $f^L(t)=(f_{k}^L (t): k\in \mathbb{N}_0)$. 
Note that $f^L (t)$ are probability distributions on $\mathbb{N}_0$ for all $t\geq 0$, and with \eqref{initialcon0b} we have $f^L_k (0)\to f_k (0)$ for all $k\geq 0$. 
Theorem \ref{thmfactorization} then implies in particular also convergence $f^L_k (t)\to f_k (t)$ for all $t\geq 0$.\\

Suppose a further symmetry assumption on the initial conditions, i.e.\ for each $L\geq 1$,
\begin{equation}
\label{initialcon1}
\textrm{the distribution of}\quad\{\eta_x (0):x\in\Lambda \}\quad\textrm{is invariant under permutation of lattice sites}\ .
\end{equation}
Then by symmetry of the dynamics, this holds also for the full process $(\eta (t):t\geq 0)$ and in particular at all fixed times $t\geq 0$. 
Then the weak law of large numbers for the empirical measures in Theorem \ref{thmfactorization} implies that for all $m\geq 1$, and $t\geq 0$ as $L\to \infty$
\[
(\eta_1(t),\eta_2(t),\ldots ,\eta_m(t))\quad\mbox{converge weakly to iidrv's with distribution } f(t)\ .
\]
This is a standard result in propagation of chaos and a recent exposition of a proof can be found in \cite{daipra2017}. 
Since the law of large numbers in Theorem \ref{thmfactorization} holds not only for time marginals but for the full process, we can also establish propagation of chaos on the level of processes.

\begin{cor}
\label{cor1}
Consider the process with generator \eqref{eq:GenMis} and conditions as in Theorem \ref{thmfactorization} together with \eqref{initialcon1}. Propagation of chaos holds, i.e. for all $m\geq 1$, and $T\geq 0$ as $L\to \infty$ the finite dimensional processes $\big( (\eta_1(t),\eta_2(t),\ldots ,\eta_m(t)):t\in [0,T]\big)$ converge weakly on path space to independent, identical birth death chains on $\N_0$ with distribution $f(t)$ and master equation given by (\ref{eq:mis_diff_f_k}).
\end{cor}

As we discuss in detail in Section \ref{sec:cons}, $\frac{d}{dt}\sum_{k\geq 0}f_{k}(t)=0$, and with normalized initial condition $f(0)$ the limit equation \eqref{eq:mis_diff_f_k} is indeed the master equation of a birth death chain with state space $\mathbb{N}_0$, birth rate $\sum_{l\geq 0}c(l,k)f_l(t)$ and death rate $\sum_{l\geq 0}c(k,l)f_l(t)$. Note that the chain and its master equation are non-linear since the birth and death rates depend on the distribution $f(t)$. 
The symmetry \eqref{initialcon1} for the initial conditions is for example fulfilled by conditional product measures as discussed in Section \ref{sec:ergo}. 
Corollary \ref{cor1} in particular implies convergence for the single site process with state space $\mathbb{N}_0$,
\begin{equation}\label{ssd}
(\eta_x (t):t\geq 0)\quad\mbox{for any fixed }x\in\Lambda\;\mbox{(with $\Lambda$ big enough)}.
\end{equation}


\section{Proof of the main result\label{sec:proof}}

The proof follows a standard approach. We first establish existence of limit processes via a tightness argument, then characterize all limits as solutions of \eqref{eq:mis_diff_f_k} and finally show that \eqref{eq:mis_diff_f_k} has a unique global solution for a given initial condition.

As a first step we collect some useful results on moments and establish a time-dependent bound on the second moment, making use of the conservation of mass as given in \eqref{alpha}.
Recall the definition \eqref{eq:f^L} of expected empirical measures, and note that their time evolution is given by
\begin{equation}
\frac{d}{dt}f_k^L (t) =\frac{d}{dt}\mathbb{E} \big[ F_k ({\eta}(t) )\big] =\mathbb{E} \big[(\mathcal{L} F_k )({\eta}(t) )\big]\ .
\label{evo}
\end{equation}
As usual this equation is not closed, since the right-hand side is not a function of $f^L(t)$ due to correlations. 
For any integer $n\geq 0$ denote the $n$-th moment by
\begin{equation}\label{lmoment}
m_n^L (t):=\E^L \Big[\frac{1}{L}\sum_{x\in\Lambda} \big(\eta_x (t)\big)^n\Big] =\sum_{k\geq 0} k^n f_k^L (t)\ .
\end{equation}
By \eqref{initialcon0b}, we have $m_1^L (0)\to\rho$ and $m_2^L (0)\to m_2 (0)<\infty$. The uniform conditions \eqref{initialcon0c} on the moments further imply for all $L\geq 1$ that $m_2^L (0)\leq \alpha_2$, and with conservation of mass \eqref{alpha} we have $m_1^L (t)\leq\alpha_1$ for all $t\geq 0$, while the second moment can grow in time.

\begin{lem}\label{lem}
Under the conditions of Theorem \ref{thmfactorization}, there exists a constant $C>0$ such that
\begin{equation}\label{eq:lem}
m_2^L (t)\leq (\alpha_2 +Ct)e^{Ct}\quad\mbox{for all }t\geq 0\mbox{ and }L\geq 1\ .
\end{equation}
\end{lem}
\begin{proof}
With the generator \eqref{eq:GenMis} we get for each $x\in\Lambda$, choosing $g(\eta )=\eta_x^2$,
\[
\Lcal \eta_x^2 =\frac{1}{L-1}\sum_{y\neq x}\Big( c(\eta_y ,\eta_x )(1+2\eta_x )+c(\eta_x ,\eta_y )(1-2\eta_x )\Big)\ .
\]
Using sublinearity \eqref{eq:lip} of the rates and the evolution equation \eqref{evo} we get after summation over $x$
\[
\frac{d}{dt} m_2^L (t)=\frac{1}{L}\sum_{x\in\Lambda} \E^L [\Lcal \eta_x^2 (t)]\leq \E^L\bigg[\sum_{k,l\geq 0}C_1 k(l+C_2 )(2+k+l)F_k^L (\eta (t))F_l^L (\eta (t))\bigg]\ .
\]
Using uniform bounds \eqref{alpha} for first moments and taking expectations of second moments we get
\[
\frac{d}{dt} m_2^L (t)\leq 2C_1 \alpha_1 (\alpha_1 +C_2 ) \big( 1+m_2^L (t)\big)\quad\mbox{for all }t\geq 0\mbox{ and }L\geq 1\ .
\]
With Gronwall's Lemma this implies
\[
m_2^L (t)\leq \big( m_2^L (0)+Ct\big) e^{Ct}
\]
with $C=2C_1 \alpha_1 (\alpha_1 +C_2 )$, and \eqref{eq:lem} follows with the bound on $m_2^L (0)$ from \eqref{initialcon0c}.
\end{proof}

\subsection{Existence of limit processes\label{sec:existence}}
\begin{prop}
\label{propexistence}
Consider the process with generator \eqref{eq:GenMis} and conditions as in Theorem \ref{thmfactorization}. For each $h$, denote by $\mathbb{Q}_h^L$ the measure of the process $t \mapsto H(\eta (t)):=\big\langle F^L (\eta (t)),h\big\rangle$ on path space $D_{[0,\infty)}(\R )$, which is the image measure of $\mathbb{P}^L$ under the mapping $\eta\mapsto \big\langle F^L (\eta),h\big\rangle$. Then $\mathbb{Q}_h^L$ is tight as $L \to \infty$.
\end{prop}
\begin{proof} 
Using a version of Aldous' criterion to establish tightness for $\mathbb{Q}_h^L$ (cf. Theorem 16.10 in \cite{billingsley2013convergence}), it suffices to show that for all $t\geq 0$
\begin{equation}
\label{eq:tight1}
\lim_{a\to \infty}\limsup_{L\to\infty} \sup_{{\zeta}\in\Omega_\alpha} \mathbb{P}_\zeta^L \big[ |H(\eta (t))|\geq a\big] =0,
\end{equation}
and that for any $\epsilon>0$
\begin{equation}
\label{eq:tightness}
\lim_{\delta\to 0^+}\limsup_{L\to \infty}\sup_{t<\delta}\sup_{{\zeta}\in\Omega_\alpha} \mathbb{P}_{{\zeta}}^L\big[ |H(\eta (t))-H(\zeta )|>\epsilon\big] =0.
\end{equation}
Here ${\zeta}\in\Omega_\alpha$ denotes a fixed initial condition of the full process satisfying \eqref{initialcon0c}, and $\mathbb{P}_{{\zeta}}^L$ the corresponding path measure.

Since $\big|\big\langle F^L (\eta),h\big\rangle\big|\leq \| h\|_\infty$ is uniformly bounded in $L$ and $\eta\in\Omega$, \eqref{eq:tight1} follows easily from Markov's inequality,
\[
\mathbb{P}_\zeta^L \big[ |H(\eta (t))|\geq a\big]\leq \frac{\mathbb{E}_\zeta^L \big[ |H(\eta (t))|\big]}{a}\leq \frac{\| h\|_\infty}{a}\quad\mbox{for all }L\geq 1\mbox{ and }\zeta\in\Omega_\alpha\ .
\]


Now fix $\delta> 0$ and consider $t<\delta$. By It\^{o}'s formula, we have
\begin{equation}
\label{eq:itoeta}
H({\eta}(t))-H(\zeta )=\int_0^t \Lcal H({\eta} (s))\, ds +M_h (t)\ ,
\end{equation}
where $(M_h (t) : t>0)$ is a martingale with quadratic variation given by integrating the 'carr\'e du champ' operator 
\begin{equation}
\label{eq:mart}
[M_h ](t)=\int_0^t \big[\mathcal{L}H^2-2H\mathcal{L}H\big] ({\eta}(s))ds\ .
\end{equation}
Using again Markov's inequality in \eqref{eq:tightness} we have to bound
\begin{equation}\label{tobound}
\E^L_\zeta \Big[\big| H({\eta}(t))-H(\zeta )\big|\Big]\leq \int_0^t \E^L_\zeta \big[|\Lcal H({\eta} (s))|\big]\, ds +\E^L_\zeta \big[ [M_h ](t)\big]^{1/2}\ ,
\end{equation}
where we have used H\"older's inequality for the martingale term and $\E^L_\zeta \big[ M_h^2 (t)\big] =\E^L_\zeta \big[ [M_h ](t)\big]$.

To compute $\Lcal H(\eta )$, we first recall that 
\begin{equation}
\mathcal{L}F_k({\eta})
=\frac{1}{L}\sum_{x\in \Lambda} \mathcal{L}\delta_{\eta_x,k} \,
\end{equation}
and applying the generator (\ref{eq:GenMis}) with $q(x,y)=\frac{1}{L-1}$, we get
\begin{align*}
\mathcal{L}\delta_{\eta_x,k}
&=\delta_{\eta_x,k+1}\sum_{y\neq x}\frac{1}{L-1}c(k+1,\eta_y)+\delta_{\eta_x,k-1}\sum_{y\neq x}\frac{1}{L-1}c(\eta_y,k-1)\\
&\quad -\delta_{\eta_x,k} \sum_{y\neq x} \frac{1}{L-1}\big(c(k,\eta_y)+c(\eta_y,k)\big) .
\end{align*}
Therefore, using the convention $F^L_{-1}\equiv 0$ in the following,
\begin{align}
\mathcal{L}H({\eta})
&=\sum_{k\geq 0} h_k \frac{1}{L}\sum_{x\in\Lambda} (\mathcal{L}\delta_{\eta_x,k}) \nonumber\\
&=\sum_{k\geq 0} h_k \bigg[ F^L_{k{-}1}({\eta})\sum_{l\geq 1}c(l,k{-}1)F^L_l({\eta}) +F^L_{k{+}1}({\eta})\sum_{l\geq 0}c(k{+}1,l)F^L_l({\eta})\nonumber\\
&\qquad\qquad -F^L_k({\eta}) \sum_{l\geq 0} \big(c(k,l)+c(l,k)\big)F^L_l({\eta}) \bigg] (1+1/L)+\Delta_L (\eta )
\label{compu}
\end{align}
which will give the integral term in It\^{o}'s formula \eqref{eq:itoeta}. 
The factor $(1+1/L)$ results from replacing $1/(L-1)$ by $1/L$, and the additive error originates from diagonal terms in summations over $y$ and is of the form
\[
\Delta_L (\eta )=\frac{1}{L-1}\sum_{k\geq 0} h_k \Big( F^L_{k+1} (\eta ) c(k{+}1,k{+}1)+F^L_{k-1} (\eta )c(k{-}1,k{-}1)-2F^L_{k} (\eta )c(k,k)\Big)\ .
\]
Using \eqref{eq:lip}, we have $c(k,k)\leq C_1 k(k+C_2 )$ and therefore
\begin{equation}\label{ebound}
\E^L_\zeta \big[|\Delta_L (\eta_s )|\big] \leq \frac{4}{L-1} \| h\|_\infty C_1 (m_2^L (s) +C_2 \alpha_1 )\ .
\end{equation}
The main contributions in \eqref{compu} can be bounded analogously by
\begin{equation}
\label{eq:paff}
4 \| h\|_\infty C_1 \alpha_1 \left(\alpha_2 +C_2 \right) (1+1/L)\ ,
\end{equation}
which holds uniformly in ${\eta}$. 
Collecting both bounds, we get with Lemma \ref{lem} and an appropriately chosen $C>0$ for the first term in \eqref{tobound}
\begin{equation}
\label{eq:limint}
\int_0^t \E^L_\zeta \big[|\Lcal H({\eta} (s))|\big]\, ds\leq  \delta\| h\|_\infty \Big( 4C_1 \alpha_1 \left(\alpha_1 +C_2\right) +\frac{C}{L}(1+\delta )e^{C\delta}\Big)\to 0 \,
\end{equation}
as $\delta\to 0$, which holds 
uniformly in ${\zeta}\in \Omega_\alpha$ and $L\geq 1$.


It remains to estimate the second term in \eqref{tobound} by the quadratic variation \eqref{eq:mart}. We need to compute terms of the form
\begin{align}
(\mathcal{L}F^L_kF^L_l)({\eta}) 
&=\frac{1}{L^2}\sum_{x,y\in \Lambda} \mathcal{L}\left(\delta_{\eta_x,k} \delta_{\eta_y,l}\right)\nonumber\\
&=\frac{1}{L^2}\sum_{x\in\Lambda\atop y\neq x}
\left(\delta_{\eta_y,l}\mathcal{L}\delta_{\eta_x,k}+\delta_{\eta_x,k}\mathcal{L}\delta_{\eta_y,l}\right)\nonumber\\
&\quad + \frac{1}{L^2}\sum_{x\in \Lambda} \Big[ \delta_{k,l}\mathcal{L}\delta_{\eta_x,k} +(1-\delta_{k,l})(\delta_{\eta_x,k}\mathcal{L}\delta_{\eta_x,l}+\delta_{\eta_x,l}\mathcal{L}\delta_{\eta_x,k}) \Big] .
\end{align}
This leads to
\begin{align}
(\mathcal{L}H^2)({\eta})
&=\sum_{k,l} h_kh_l \bigg[ 2\frac{1}{L^2}\sum_{x\in\Lambda\atop y\neq x}\delta_{\eta_y,l}\mathcal{L}\delta_{\eta_x,k}+\delta_{k,l}\frac{1}{L^2}\sum_{x\in\Lambda} \mathcal{L}\delta_{\eta_x,k}\nonumber\\
&\qquad\qquad\quad +2(1-\delta_{k.l})\frac{1}{L^2}\sum_{x\in\Lambda} \delta_{\eta_x,l}\mathcal{L}\delta_{\eta_x,k}\bigg] ,
\end{align}
and 
\begin{align}
2H(\mathcal{L}H)({\eta})
&=2\sum_{k,l\geq 0}h_kh_l\bigg[\frac{1}{L^2}\sum_{x\in\Lambda\atop y\neq x} \delta_{\eta_y,l}\mathcal{L}\delta_{\eta_x,k}+\frac{1}{L^2}\sum_{x\in\Lambda} \delta_{\eta_x,l}\mathcal{L}\delta_{\eta_x,k}\bigg] .
\end{align}
As usual, the leading order contributions to the carr\'e du champ cancel and we are left with diagonal terms
\begin{align}
\label{eq:mart2}
(\mathcal{L}H^2)-2H(\mathcal{L}H)
&=\sum_{k,l\geq 0} h_k h_l\frac{1}{L^2}\sum_{x\in\Lambda}\big[\delta_{k,l}\big( \mathcal{L}\delta_{\eta_x,k}-2\delta_{\eta_x,l}\mathcal{L}\delta_{\eta_x,k}\big)\big] \nonumber \\
&=\sum_{k\geq 0} h^2_k \frac{1}{L^2}\sum_{x \in\Lambda}(1-2\delta_{\eta_x,k})\mathcal{L}\delta_{\eta_x,k} \nonumber \\
&=\frac{1}{L} \sum_{k\geq 0} h^2_k \bigg[ F^L_{k{-}1}\sum_{l\geq 0}c(l,k{-}1)F^L_l+F^L_{k{+}1}\sum_{l\geq 0}c(k{+}1,l)F^L_l \nonumber \\
&\qquad\qquad\quad\ +F^L_k\sum_{l\geq 0} \big(c(k,l)+c(l,k)\big)F^L_l \bigg] (1+1/L)+\Delta'_L \ ,
\end{align}
where we have suppressed arguments $(\eta )$ to simplify notation. Error terms have the same origin as in \eqref{compu} and can be estimated completely analogously. 
With bounds corresponding to \eqref{ebound} for $\Delta'_L$ and \eqref{eq:paff} this implies for the quadratic variation part in \eqref{tobound} 
\begin{equation}
\label{eq:limquadratic}
\E^L_\zeta \big[[M_h ](t)\big]\leq t\| h\|_\infty^2 \frac{1}{L} \Big( 4C_1 \alpha_1 \left(\alpha_1 +C_2\right) +\frac{C}{L}(1+t )e^{Ct}\Big)
 \to 0\ ,
\end{equation}
with $t<\delta$ as $\delta\to 0$. This holds again uniformly in ${\zeta}\in \Omega_\alpha$ and $L\geq 1$, and finishes the proof.
\end{proof}

By Prohorov's theorem, the tightness result in Proposition \ref{propexistence} implies the existence of limit points of the sequence $\big(\langle F^L (t),h\rangle :t\geq 0\big)$ in the usual topology of weak convergence on path space. 
By linearity of $h\mapsto \langle F^L(\eta(.)),h \rangle $, convergence along a subsequence for a given test function $h$ implies convergence for all bounded $h$. This establishes existence of limit processes $(f(t): t\geq 0)$ which may not be unique and still be random at this point.

\subsection{Characterization of limit points\label{sec:limit}}

Due to the extra factor $1/L$ in the estimate of the quadratic variation in (\ref{eq:limquadratic}), we see that in fact we have
\[
\E_\zeta^L \big[ [M_h ](t)\big] =\E_\zeta^L \big[ M_h^2 (t)\big]\to 0\quad \mbox{as }L\to\infty\quad\mbox{for all }t\geq 0.
\]
This holds uniformly in the initial condition ${\zeta}\in\Omega_\alpha$, and implies that $M_h (t)\to 0$ in (\ref{eq:itoeta}) in $L^2$-sense for all $t\geq 0$. Therefore, inserting the limit points $(f(t): t\geq 0)$ in (\ref{compu}), they solve the following deterministic equation
\begin{align}\label{converge}
\langle f(t),h\rangle {-}\langle f(0),h\rangle  
&=\int_0^t \sum_{k,l\geq 0}h_k c(k{+}1,l)f_l(s)f_{k+1}(s){+}\sum_{k,l\geq 0} h_k c(l,k{-}1) f_l(s)f_{k-1}(s) \nonumber \\
&\quad-\bigg(\sum_{k,l\geq 0}h_k c(k,l)f_l(s)+\sum_{k,l\geq 0} h_k c(l,k)f_l(s)\bigg)f_{k}(s)\, ds
\end{align}
for all bounded $h:\N_0\to \R$.
This is equivalent to \eqref{eq:mis_diff_f_k}, which has a unique global solution as shown in the next subsection, thereby identifying the limit $(f(t): t\geq 0)$ and establishing the law of large numbers.

\subsection{Uniqueness \label{sec:uniqueness}}

We consider solutions of \eqref{eq:mis_diff_f_k}, $f(t)=(f_k(t): k\in \mathbb{N}_0)$ which are limit points of the sequence $f^L (t)$. It has been shown recently in \cite{esenturk2017} independently, that under our conditions \eqref{eq:mis_diff_f_k} has a unique solution.
This work is based on the classical proof in \cite{ball1986becker} and for completeness we will present a condensed version of the proof of uniqueness, whereas existence of the solution is then implied by our limit result. 
In analogy to \eqref{lmoment}, we denote for all $n\in\N_0$ the $n^{th}$ moment of a solution $f(t)$ by
\begin{equation}
\label{moments}
m_n (t)=\sum_{k\geq 0} k^n f_k(t).
\end{equation}

\begin{prop}
Let $t\mapsto f(t)$ be a solution to \eqref{eq:mis_diff_f_k} with moments
\[
m_0 (0)=1\ ,\quad m_1 (0)=\rho <\infty\quad\mbox{and}\quad m_2 (0)<\infty\ .
\]
Then, $t \mapsto f(t)$ is unique. 
\end{prop}
\begin{proof}
As explained in Section \ref{sec:cons}, any solution with the above properties has conserved moments
\[
m_0 (t)=1\quad\mbox{and}\quad m_1 (t)=\rho\quad\mbox{for all }t\geq 0\ ,
\]
and since we assume $m_2 (0)<\infty$, as is shown in (\ref{eq:gwall}) later it has also bounded second moment for all compact time intervals $t\in [0,T]$. 
	
Suppose $f$ and $\hat{f}$ are two solutions of \eqref{eq:mis_diff_f_k} with the above properties and $f(0)=\hat{f}(0)$. We write $\Delta_k (t):=f_k (t)-\hat f_k (t)$ and establish a Gronwall estimate for
\begin{equation}
\label{theta}
\theta (t):=\sum_{k\geq 0} (k+1) |\Delta_k (t)|
\end{equation}
for all times $t\in [0,T]$ in an arbitrary compact time interval. Due to conservation of zeroth and first moment $\theta (t)\leq 2(1+\rho )$ is well defined for all $t\geq 0$, and of course $\theta (0)=0$.

Following \cite{ball1986becker}, note that for any uniformly continuous function $\psi (t)$ also $|\psi (t)|$ is uniformly continuous, and we have
\[
\frac{d}{dt} |\psi (t)| =\sgn \psi (t) \frac{d\psi}{dt} (t)\quad \mbox{for almost all }t\geq 0\ ,
\]
where for $\psi\in\R$ we set $\sgn\psi$ equal to $-1,\ 0$ or $1$ according to $\psi <0$, $\psi =0$ and $\psi >1$, respectively. 
Using the convention $c(k,l)=0$ if $l<0$ or $k\leq 0$, and suppressing the time variables of $f$ to simplify notation, we get from \eqref{eq:mis_diff_f_k}
\begin{align}
\frac{d}{dt} \theta(t) 
=&\sum_{k\geq 0} (k{+}1)\sgn\Delta_k(t)  \nonumber\\
& \quad\bigg (\sum_{l\geq 0} \Big[ c(l,k{-}1)f_l f_{k-1}  +c(k{+}1,l)f_l f_{k+1} -\big( c(l,k)+c(k,l)\big) f_l f_k\Big]\nonumber\\
& \quad\quad+\sum_{l\geq 0} \Big[ c(l,k{-}1)\hat f_l \hat f_{k-1}  +c(k{+}1,l)\hat f_l \hat f_{k+1} -\big( c(l,k)+c(k,l)\big) \hat f_l \hat f_k\Big]\bigg)\nonumber\\
=& \sum_{k,l\geq 0} \Big[ c(l,k) ( s_{k+1} -s_{k} )+c(k,l) (s_{k-1}-s_{k})\Big] \Big( f_k \Delta_l +\hat f_l \Delta_k \Big)\ ,
\end{align}
where we use the shorthand $s_k =(k+1)\sgn \Delta_k(t)$ with the convention $s_{-1}=0$ 
in the last line.
For the terms involving $\hat{f}_l \Delta_k$, we can take $\sgn\Delta_k$ out of the first bracket, and use $\sum_{k\geq 0} |\Delta_k| \leq \theta(t)$ 
and $f_l \geq 0$ to get the upper bound
\[
\sum_{k,l\geq 0} \big[ c(l,k)-c(k,l) \big] \hat{f}_l |\Delta_k |\leq \sum_{k,l\geq 0} 2 C_1 (C_2 +k)(C_2 +l) \hat{f}_l |\Delta_k |\leq 2C_1 C_2 (C_2 +\rho) \theta(t)
\]
with (\ref{eq:lip}), and assuming without loss of generality,\ that $C_1 ,C_2\geq 1$.\\
For the term involving $f_k \Delta_l$, since $\Delta_l$ does not have a fixed sign, we need to estimate the sum in the square brackets to get the upper bound
\[
\sum_{k,l\geq 0} (2k+3)\big( c(l,k)+c(k,l)\big) f_k (t)|\Delta_l (t)|\leq C \big( 1+\rho +m_2 (t)\big) \theta (t),
\]
analogously to before for a large enough constant $C>0$. From (\ref{eq:gwall}) in Proposition \ref{gwall} we know that $\sup_{t\in [0,T]} m_2 (t)\leq (m_2 (0)+\bar C_2 T)\, e^{\bar C_2 T}$ for any fixed $T>0$, so that we can combine both bounds to get
\[
\frac{d}{dt} \theta (t)\leq C\big( 1+\rho +(m_2 (0)+\bar C_2 T)\, e^{\bar C_2 T}\big) \theta (t)=:C_T \theta (t)
\]
for all $t\leq T$ with a suitable constant $C_T$. By Gronwall's lemma, this implies
\[
\theta (t)\leq \theta (0)e^{C_T t}=0\quad\mbox{for all }t\leq T
\]
since $\theta (0)=0$. This completes the proof since $T>0$ was arbitrary.
\end{proof}

\subsection{Proof of Corollary \ref{cor1}\label{sec:cor2}}

Propagation of chaos for processes follows from the standard result in \cite{daipra2017}, provided we lift the law of large numbers for empirical processes in Theorem \ref{thmfactorization} to a law of large numbers on the path space of birth death chains. For any fixed $T>0$, denote the path of the occupation number on site $x\in\Lambda$ by $\eta_x [0,T]=(\eta_x (t):t\in [0,T])$ and by
\[
\hat \Q^L_{[0,T]} [d\omega ]=\frac{1}{L}\sum_{x \in \Lambda} \delta_{\eta_x [0,T]} [d\omega ]\quad\mbox{the empirical measure on path space $D_{[0,T]} (\N_0 )$}\ .
\]
We have to show that these random measures converge to the path measure $\hat\Q_{[0,T]}$ of a non-linear birth death chain with (time-dependent) generator for bounded, continuous $h:\N_0 \to\R$
\begin{equation}
\Lcal_{f(t)}  h(k) :=\sum_{l\geq 0} f_l (t) \Big( c(l,k)\big( h(k+1)-h(k)\big) +c(k,l)\big( h(k-1)-h(k)\big) \Big)\ ,
\label{ssgencon}
\end{equation}
corresponding to the master equation \eqref{eq:mis_diff_f_k}. As a side remark, this generator can also be derived as the limit of the original full generator $\Lcal h(\eta_x (t))$ applied to a test function of one variable only. The marginals of $\hat \Q^L_{[0,T]}$ at given times $t$ are simply $F^L (\eta (t))$, and Theorem \ref{thmfactorization} directly implies convergence of marginals for any $t_1 <\ldots <t_m$, $m>0$
\[
\big( F^L (\eta (t_1 )),\ldots ,F^L (\eta (t_m ))\big)\to \big(f(t_1 ),\ldots ,f(t_m )\big)
\]
in a weak sense, i.e.\ integrated against any test function $h\in C(\N_0 )$. The solution $f (t)$ of \eqref{eq:mis_diff_f_k} determines the time marginals of the limiting path measure $\hat\Q_{[0,T]}$.

It remains to show convergence of finite dimensional distributions of the path measures. A weak version of \eqref{eq:mis_diff_f_k} can be written as
\begin{equation}\label{eq:lime}
\frac{d}{dt}\langle f(t),h\rangle =\langle f(t),\Lcal_{f(t)} h\rangle\ ,
\end{equation}
and as usual the time evolution of the solution is given by the propagator $T_f (t_1 ,t_2):C^b (\N_0)\to C^b (\N_0 )$ such that
\[
\big\langle f(t_2 ),h\big\rangle =\big\langle f(t_1 ) ,T_f (t_1 ,t_2) h\big\rangle\quad\mbox{for any }0\leq t_1 <t_2 \leq T\quad\mbox{and}\quad h\in C^b (\N_0 )\ .
\]
Informally, one often writes $T_f (t_1 ,t_2)=\exp \big(\int_{t_1}^{t_2} \Lcal_{f(s)} ds\big)$, and due to non-linearity of \eqref{eq:lime} the propagator does not form a semigroup and depends on the solution $f$ (see e.g.\ \cite{kolokoltsov} for details). 
As usual, the path measure $\hat\Q_{[0,T]}$ is fully characterized by the propagator via its finite dimensional distributions. Denoting a canonical path by $\omega\in D_{[0,T]} (\N_0 )$, we have for example for a two-point event
\[
\hat{\Q}_{[0,T]} \big[\omega (t_1)=k_1 ,\omega (t_2)=k_2 \big] =\big\langle f(0),T_f (0,t_1 )\delta_{k_1} \big\rangle\big\langle \delta_{k_1} ,T_f (t_1 ,t_2) \delta_{k_2}\big\rangle\ .
\]
The propagator for the random path measures $\hat \Q^L_{[0,T]}$ is given by the semigroup $T(t):C^b (\Omega)\to C^b (\Omega)$ of the original process via
\[
\E\big[\langle F^L (\eta (t_2 )),h\rangle \big| \mathcal{F}_{t_1}\big] =\big\langle T(t_2 -t_1) F^L (\eta (t_1 )),h\big\rangle\ ,
\]
where $(\mathcal{F}_t :t\geq 0)$ is the natural filtration of the process $(\eta (t):t\geq 0)$. Convergence of the full processes $\big( F^L (\eta (t)):t\geq 0\big)$ to $\big( f(t):t\geq 0\big)$ in Theorem \ref{thmfactorization} implies in particular convergence of propagators, and therefore a weak law of large numbers for path space empirical measures $\hat \Q^L_{[0,T]} \to \hat \Q_{[0,T]}$ for any fixed $T>0$, as required.

\section{Properties of solutions \label{sec:properties}}

It has recently been shown independently \cite{esenturk2017} in a more general setting that \eqref{eq:mis_diff_f_k} with sublinear rates \eqref{eq:lip} has a unique, global solution, and that $t\mapsto f_k (t)$ is continuously differentiable for all $k\in\N_0$ and $t\geq 0$. In the following we summarize a few further important properties of the solutions to the mean-field equation including conserved quantities and stationary distributions, which also leads to several interesting open questions beyond the scope of this paper.

\subsection{Conserved quantities and moments \label{sec:cons}}

For any solution of \eqref{eq:mis_diff_f_k} the $f_k(t)$ are in particular limits of the expectations $f_k^L(t) \in [0,1]$ defined in \eqref{eq:f^L}, so we have $f_k(t)\in[0,1]$ for all $k\geq 0, t\geq 0$. Since for all $t\geq 0$ and $L\geq 1$ we have $\sum_{k\geq 0} f_k^L (t)=1$ and with \eqref{initialcon0c} and \eqref{alpha} also $\sum_{k\geq 0} kf_k^L (t)\leq\alpha_1$, we further know a-priori by Fatou's Lemma that
\begin{equation}\label{eq:fatou}
\sum_{k\geq 0} f_k (t)\leq 1\quad\mbox{and}\quad \sum_{k\geq 0} kf_k  (t)\leq\alpha_1 \quad\mbox{for all }t\geq 0\ .
\end{equation}
With Corollary \ref{cor1} the limiting mean-field equation \eqref{eq:mis_diff_f_k} is also the master equation of the non-linear birth death chain $(X_t :t\geq 0)$ on $\mathbb{N}_0$ with generator
\begin{equation}\label{bdgen}
\mathcal{L}_{f(t)}h(k)=\sum_{l\geq 0} c(k,l)f_l(t)\big( h(k{-}1)-h(k)\big) +\sum_{l\geq 0} c(l,k)f_l(t)\big( h(k{+}1)-h(k)\big)\ ,
\end{equation}
where $c(0,l)=0$ for all $l\geq 0$. Under additional symmetry assumptions for the initial condition \eqref{initialcon1}, this is equal to the limit dynamics of the single site process \eqref{ssd}, and the time dependent birth rates $\beta_k (t)$ and death rates $\mu_k (t)$ can be written as
\begin{equation}\label{bdrates}
\beta_k (t)=\sum_{l\geq 0}c(l,k) f_l (t)\quad\mbox{and}\quad\mu_k (t)=\sum_{l\geq 0}c(k,l) f_l (t).
\end{equation}
It is clear from sublinearity of the jump rates (\ref{eq:lip}) that the generator (\ref{bdgen}) is defined for at least all bounded functions $h:\N_0 \to\R$ (in fact also functions with at most linear growth rate), which is sufficient to fully characterize the adjoint operator $\mathcal{L}_{f(t)}^{\dagger}$ on probability measures on $\N_0$. This non-linear operator then describes the right-hand side of the master equation \eqref{eq:mis_diff_f_k} which can be written as
\[
\frac{d}{dt}f(t)=\mathcal{L}_{f(t)}^{\dagger}f(t).
\]
$f(t)$ is indeed a probability distribution on $\mathbb{N}_0$ for all $t\geq 0$ since we have 
\[
\mathcal{L}_{f(t)}1=0\quad\mbox{and therefore }\quad m_0(t)=m_0(0)=1\ ,
\]
i.e.\ conservation of the $0$-th moment \eqref{moments}. Also, as usual for birth death chains $\mathcal{L}_{f(t)}k=\beta_k (t)-\mu_k (t)$, which leads to
\[
\frac{d}{dt}m_1(t)=\sum_{k\geq 0}f_k(t)\mathcal{L}_{f(t)}k=\sum_{k\geq 0}\sum_{l\geq 0} f_k(t) f_l(t) \big( c(l,k)-c(k,l)\big) =0,
\]
since with (\ref{eq:fatou}) the sum converges absolutely. 
This implies that the expectation is conserved for the chain $(X_t :t\geq 0)$, and with \eqref{initialcon0}
\[
m_1(t)=m_1(0)=:\rho>0\quad\mbox{for all }t\geq 0\ ,
\]
which corresponds to the asymptotic density $\rho$ in the original particle system. Note, however, that $(X_t : t\geq 0)$ is not a martingale since $\mathcal{L}_{f(t)}k\neq 0$, and the conservation of $m_1$ results from the non-linearity of the process.

Analogously to Lemma \ref{lem}, conservation of the first moment implies a Gronwall estimate for higher order integer moments for the solution $f(t)$.
\begin{prop}\label{gwall}
Assume that $m_n (0)<\infty$ for some integer $n\geq 2$. Then there exists a constant $\bar C_n$ such that
\begin{equation}\label{eq:gwall}
m_n (t) \leq \big( m_n (0)+\bar C_n t\big) e^{\bar C_n t}\quad\mbox{for all }t\geq 0\ .
\end{equation}
\end{prop}

\begin{proof}
Note that $(k\pm 1)^n -k^n =p^\pm_{n-1} (k)$ is a polynomial of degree $n-1$, which implies with (\ref{bdgen}) and sublinear rates (\ref{eq:lip}) that
\[
\mathcal{L}_{f(t)} k^n =\sum_{l\geq 0} c(k,l)f_l(t)p^-_{n-1} (k)+\sum_{l\geq 0} c(l,k)f_l(t)p^+_{n-1} (k)\leq C\sum_{l\geq 0} l f_l(t) p_n (k)
\]
for some constant $C>0$ and polynomial $p_n (k)$ of degree $n$. Since $\sum_{l\geq 0} l f_l(t)=\rho <\infty$, and $m_{n} (t)\leq m_{n+1} (t)$ for all $n\geq 1$, this implies with $m_0 (t)=1$ that for some constant $\bar C_n >0$
\[
\frac{d}{dt} m_n (t)=\frac{d}{dt} \E \big[ \Lcal_{f(t)} X_t^n\big]\leq \bar C_n \big( 1+m_n (t)\big)\ .
\]
The result then follows by Gronwall's Lemma.	
\end{proof}

\subsection{Stationary distributions}

Note that with \eqref{bdrates} it follows immediately that $\bar f_k :=\delta_{0,k}$ is a stationary distribution for the limiting single site chain $(X_t :t\geq 0)$, but in general $0$ is not an absorbing state as long as $f_k (0)>0$ for some $k>0$. 
By assumption \eqref{cassum} on the rates $c$ the chain is further irreducible unless $f(0)$ is degenerate, but we will discuss below how the additional conserved first moment leads to non-uniqueness for the stationary distribution.

Under certain conditions on the jump rates we can identify a class of stationary distributions via the original particle system. A misanthrope-type process with generator (\ref{eq:GenMis}) on the complete graph has a family of stationary product measures $\nu_{\phi}$, if and only if
\begin{equation}
\label{eq:conmisrate}
\frac{c(k,l)}{c(l{+}1,k{-}1)}=\frac{c(k,0)c(1,l)}{c(l{+}1,0)c(1,k{-}1)}\;\; \mbox{ for all  }k\geq 1,l\geq 0. 
\end{equation}
This is well known (see e.g. \cite{evans2014condensation,cocozza1985processus,fajfrova2016,rafferty2015monotonicity}), also for more general translation invariant dynamics under additional conditions on $c$. The marginals are given explicitly by
\begin{equation}
\label{eq:marg}
\nu_{\phi} [\eta_{x}=n]=\frac{1}{z(\phi)}w(n)\phi^{n}\quad\mbox{with}\quad w(n)=\prod_{k=1}^n \frac{c(1,k{-}1)}{c(k,0)},
\end{equation}
which are normalized by the partition function
\[
z(\phi):=\sum_{n=0}^{\infty}w(n)\phi^{n}\ .
\]
The parameter $\phi\geq 0$ is the fugacity controlling the average particle density
\begin{equation}
\label{eq:R}
R(\phi):=\sum_{n=0}^{\infty}n \nu_{\phi} [\eta_{x}=n]=\phi\partial_{\phi}\log(z(\phi)),
\end{equation}
which is a monotone increasing function of $\phi$ with $R(0)=0$. 
These distributions exist for all $\phi \in \mathcal{D} :=\{\phi \geq 0 : z(\phi)< \infty\}$. The domain is of the form $\mathcal{D} =[0,\phi_c]$ or $[0,\phi_c)$, where 
\[
\phi_c=(\limsup_{n\to\infty} w(n)^{1/n})^{-1}
\]
is the radius of convergence of $z(\phi)$, and we denote by
\begin{equation}
\rho_c=R({\phi}_c)\in [0,\infty]
\end{equation}
the maximal density for the family of product measures. 
If $\mathcal{D} =[0,\phi_c)$ then $\rho_c=\infty$, and if $\rho_c <\infty$ the model exhibits a condensation transition (see e.g.\ \cite{chleboun2014condensation}), which we discuss in more detail in Section \ref{sec:examples}.

\begin{prop}
	The single site marginal
	\begin{equation}
	\bar f^{\phi}_k :=\nu_{\phi}[\eta_x=k]\quad\mbox{for each }\phi\in\mathcal{D}
	\end{equation}
	is a stationary solution to \eqref{eq:mis_diff_f_k}, and a reversible distribution for $(X_t :t\geq 0)$.
\end{prop}

\begin{proof}
From \eqref{eq:marg}, we have the relation
\begin{equation}
\label{eq:relation}
\frac{c(k,0)}{c(1,k{-}1)}\bar f^{\phi}_k=\phi \bar f^{\phi}_{k{-}1}\quad\mbox{for all }k\geq 1\;\mbox{and }\phi\in \mathcal{D}\ .
\end{equation}
Then the detailed balance equations for the single site dynamics with birth an death rates \eqref{bdrates} follow as
\begin{align*}
\bar f^{\phi}_{k{+}1}\sum_{l\geq 0} c(k{+}1,l) \bar f^{\phi}_l
&=\sum_{l\geq 1} c(k{+}1,l{-}1) \bar f^{\phi}_{l{-}1} \frac{c(1,k)}{c(k{+}1,0)}\phi \bar f^{\phi}_k\\
&=\sum_{l\geq 1} c(l,k)\frac{c(1,l{-}1)}{c(l,0)}\phi \bar f^{\phi}_{l{-}1}\bar f^{\phi}_k\\
&=\bar f^{\phi}_k\sum_{l\geq 1} c(l,k)\bar f^{\phi}_l \ ,
\end{align*}
where in the first and last equality we use \eqref{eq:relation} and in the second equality we use \eqref{eq:conmisrate}. This implies in particular that $\bar f^{\phi}$ is a stationary solution of \eqref{eq:mis_diff_f_k}.
\end{proof}

Therefore, under condition \eqref{eq:conmisrate} we have an explicit stationary distribution for each value $\rho =m_1 (0)$ of the conserved first moment provided it is not larger than $\rho_c$, given by $\bar f^\phi$ with $\phi\in \mathcal{D}$ such that $R(\phi)=\rho$. Note that for $\phi =0$ we have $\bar f^0_k =\delta_{0,k}$ corresponding to $\rho =0$. 
In general, due to the restriction of birth death dynamics without long-range jumps, we expect all stationary distributions $f$ to be reversible for $(X_t :t\geq 0)$ and given by solutions to the detailed balance equations
\begin{equation}\label{ssdetbal}
\bar f_k \sum_l c(l,k)\bar f_l =\bar f_{k+1} \sum_l c(k+1,l)\bar f_l \quad\mbox{for all }k\geq 0\ .
\end{equation}
These equations are non-linear and if the rates do not obey condition \eqref{eq:conmisrate} we are not aware of a method to find explicit solutions in the general case. Based on the connection to the underlying particle system, we outline a possible approach below to establish at least existence of solutions.\\

Consider a sequence of particle systems $(\eta (t):t\geq 0)$ with generator (\ref{eq:GenMis}) in the thermodynamic limit, i.e.\ with a deterministic number of particles $N$ where $N/L\to\rho$. For each $L,N$ the process is a finite state Markov chain which under our conditions on the jump rates is irreducible on the subset of configurations with $N$ particles. Therefore the process has a unique stationary measure $\pi_{L,N}$ which is also reversible due to symmetric dynamics on the complete graph. Due to symmetry the marginal
\begin{equation}\label{rhomar}
\bar f^L_k :=\pi_{L,N} [\eta_x =k]
\end{equation}
is also independent of $x\in\Lambda$. Provided that $\bar f^L$ converges to a probability distribution $\bar f^\rho$ as $L,N\to\infty$, $N/L\to\rho$, this limit is a reversible distribution for $(X_t :t\geq 0)$ and fulfills \eqref{ssdetbal}. 
If condition \eqref{eq:conmisrate} holds then $\bar f^\rho =\bar f^\phi$ with $\rho =R(\phi )$, provided that $\rho\leq\rho_c$. 
This follows from the equivalence of the ensembles $\pi_{L,N}$ and $\nu_{\phi}$ in the limit $L\to\infty$, which has been established for these models in great detail (see e.g.\ \cite{chleboun2014condensation} and references therein). In the absence of condition \eqref{eq:conmisrate} we still expect convergence of the marginals \eqref{rhomar} under sublinearity conditions \eqref{eq:lip} on the jump rates and possible further assumptions, but this is an open question beyond the scope of this paper. We will mention an example in Section \ref{sec:examples} where both, \eqref{eq:lip} and convergence of the marginals are violated. While the above provides an approach to establish existence of stationary distributions identified as a particular limit, it cannot address uniqueness for a given first moment $m_1 (0)$, which remains a further interesting open problem. In general, existence and uniqueness of stationary distributions for non-linear birth death chains is a challenging problem where only partial results have been obtained so far (see e.g.\ \cite{feng1992solutions}).

\subsection{Initial conditions and ergodic behaviour\label{sec:ergo}}

Consider a fixed initial condition $f(0)$ for the limit equation \eqref{eq:mis_diff_f_k} with finite density $\rho =m_1 (0)\in (0,\infty )$ and $m_2 (0)<\infty$ as in \eqref{initialcon0}. Consider product measures $\nu^L$ with marginals $\nu^L [\eta_x =.]=f(0)$ so that simply $f^L (0)=f(0)$ for all $L\geq 1$. Then \eqref{initialcon0b} holds trivially, and \eqref{initialcon0c} holds for example if $f(0)$ has only finite support, i.e. $f_k (0)=0$ for $k\geq K$ for some $K>0$. If $f(0)$ has infinite support, we still have a weak law of large numbers for the empirical first and second moment appearing in \eqref{initialcon0c}. If $f(0)$ has exponential or high enough integer moments and we choose $\alpha_1 >\rho$ and $\alpha_2 >m_2 (0)$ large enough, the probability of $\eta (0)\not\in\Omega_\alpha$ vanishes as $L\to\infty$, and can be bounded easily for specific examples. The uniform bounds \eqref{initialcon0c} are used in the main estimates of expectations of \eqref{compu} and \eqref{eq:mart2} in Section \ref{sec:existence}. Since the functions $H$ and $H^2$ involved are bounded, the proof can be adapted by a simple split of the expectation into events $\eta (0)\in\Omega_\alpha$ and $\eta (0)\not\in\Omega_\alpha$, where the second is simply estimated by its probability multiplied by the bound of the function. 


Another generic choice fulfilling all conditions \eqref{initialcon0} to \eqref{initialcon0c} and permutation invariance \eqref{initialcon1} are conditional product measures with a fixed number of particles, i.e.
\begin{equation}\label{condma}
\mu_{L,N} =\nu^L \Big[\, .\,\Big|\sum_{x \in \Lambda} \eta_x =N ,\sum_{x \in \Lambda} \eta_x^2 \leq L\alpha_2\Big]\quad\textrm{and}\quad f^L(0)=\mu_{L,N}[\eta_x=\cdot]\ .
\end{equation}
Choosing again $N/L\to\rho$ in the thermodynamic limit, $f^L (0)\to f(0)=\nu$ as $L\to\infty$ weakly due to the equivalence of ensembles \cite{chleboun2014condensation}. The additional conditioning on the second moment can be included easily since this is a typical event in the limit, provided $\alpha_2 >m_2 (0)$ and $\nu$ has appropriate tails with finite second moment as mentioned above.
A particular generic example of this form is to simply distribute $N$ particles uniformly at random, leading to binomial marginals
\begin{equation}
{N \choose k}\left(\frac{1}{L}\right)^k\left(1-\frac{1}{L}\right)^{N-k}=f^L_k(0) \to f_k(0)=\frac{\rho^k}{k!}e^{-\rho},
\end{equation}
with standard convergence to Poi($\rho$) variables as $N,L\to \infty$ and $N/L\to \rho$. As such this does not obey the uniform condition on the second moment, but choosing the arbitrary constant $\alpha_2$ large enough, binomial samples fulfill it with high probability approaching $1$ as $L\to\infty$.

Given the family of stationary measures $\bar f^\phi$ introduced in the previous Section, a natural question is that of ergodicity, i.e. for initial conditions $f(0)$ with first moment $\rho =m_1 (0)<\infty$, does $f(t)$ converge to $\bar f^\phi$ with $R(\phi)=\rho$? While contraction arguments may by possible for particular jump rates $c(k,l)$, we are not aware of general results on convergence to stationary solutions for non-linear dynamical systems that would answer this question. 
As mentioned before, on the restricted state space $\big\{\eta\in\Omega :\sum_{x \in \Lambda} \eta_x =N\big\}$ with a fixed number of particles the process $(\eta (t):t\geq 0)$ is a finite state, irreducible Markov chain, which is therefore ergodic and converges to its unique stationary distribution $\pi_{L,N}$. 
This implies
\[
f^L (t)\to \pi_{L,N} [\eta_x =.]\quad\mbox{as }t\to\infty
\]
for each finite $L$, which holds in total variation or $L^2$ distance. 
If the convergence $f^L (t)\to f (t)$ was uniformly in $t>0$, this could be used to establish ergodicity for the limit process. While we do not state explicit bounds on the distance of $f^L (t)$ and $f (t)$ in our proof, the crucial estimates are in \eqref{eq:limint}, \eqref{eq:limquadratic}, 
and lead to bounds proportional to $e^{Ct}/L$. They are clearly only directly useful for $t\ll \log L$ (in particular for all fixed $t>0$), and our proof does not provide uniform convergence. 
In general, ergodicity breaking is a well-known phenomenon in the presence of phase transitions, e.g. for the contact process uniqueness of the stationary distribution is lost in infinite volume (see e.g. \cite{Liggett1985InteractingParticleSystems} Chapter 6). 
For solutions to \eqref{eq:mis_diff_f_k}, however, we still expect ergodicity at least for $\rho \leq\rho_c$. Explicit heuristic scaling solutions for particular systems discussed in the next section suggest that ergodicity may hold even for $\rho >\rho_c$.\\

A possible approach to establish at least local stability of the stationary distribution could be to estimate relaxation times $t_{rel}$ for the underlying particle system. By standard path counting arguments (analogous to e.g.\ \cite{armendariz2015metastability} Section 4.2), they should be bounded above by a constant independently of the system size $L$ due to complete graph dynamics. The resulting error bounds of the form $e^{Ct_{rel}}/L$ would vanish in the limit $L\to\infty$ in order to establish convergence when starting close to stationarity. It is clear that far away from stationarity, e.g.\ putting all $N$ particles on a single site, mixing times have to be at least of order $L$.

%

\section{Examples of condensing particle systems\label{sec:examples}}

%

To further illustrate the relevance of our results we discuss two classes of processes of type \eqref{eq:GenMis} that exhibit condensation and have attracted significant recent research interest. Not all cases are covered by our main theorem, but we include them to illustrate the possible irregular behaviour and non-existence of solutions to \eqref{eq:mis_diff_f_k} related to gelation in growth/aggregation models as explained below. The results on stationary measures in this section have been established on a rigorous level, while dynamic results and related scaling solutions to \eqref{eq:mis_diff_f_k} have only been studied heuristically so far. Nevertheless, we think it is instructive to include them as examples of particular solutions to the limit equation, which pose interesting research problems on a rigorous level.

\subsection{Zero-range processes\label{zrp}}
For zero-range processes (ZRP) the jump rates depend only on the occupation of the departure site, and we use the notation
\[
c(k,l)=g(k)\quad\mbox{with }g:\mathbb{N}\rightarrow[0,\infty)\mbox{ such that }g(k)=0\Leftrightarrow k=0\ .
\]
This leads to the mean-field equation \eqref{eq:mis_diff_f_k} taking the form
\begin{equation}
\frac{df_{k}(t)}{dt}=g(k+1)f_{k+1}(t)+{\bar{g}}(t)f_{k-1}(t)-(g(k)+{\bar{g}}(t))f_{k}(t),
\label{eq:zrp_diff_f_k}
\end{equation}
where $\bar{g}(t)=\sum_k g(k)f_k(t)$, valid for all $k\geq 0$ with the convention $f_{-1}(t)\equiv 0$. As before this is the master equation of a birth death chain with $k$-independent birth rate $\bar{g}(t)$ and time-independent death rate $g(k)$, which have been studied in \cite{feng1992solutions}. ZRPs satisfy \eqref{eq:conmisrate} for all choices of rates $g$ and exhibit stationary product measures of the form \eqref{eq:marg}.

An interesting example is given by the bounded jump rates
\begin{equation}
\label{rates}
g(k) =  
\left\{ \begin{array}{cl}
0 & \mbox{if }k=0, \\
1+\frac{b}{k^\gamma} & \mbox{if }k\geq 1,
\end{array} \right.
\end{equation}
with parameters $b>0$ and $\gamma\in(0,1]$, which obviously fulfill \eqref{eq:lip} and our result applies. For the measures \eqref{eq:marg} we have $\phi_c=1$ and stationary weights 
\begin{align*}
w(n)&=\prod_{k=1}^n \frac{k}{k+b} \sim n^{-b} &\mbox{for }&\gamma=1,\\
w(n)&=\prod_{k=1}^n \frac{k^\gamma}{k^\gamma+b} \sim \exp\Big( -\frac{C}{1-\gamma}n^{1-\gamma}\Big)\ ,\ C>0 &\mbox{for }&\gamma\in(0,1). 
\end{align*}
The symbol $\sim$ indicates asymptotic proportionality as $n\to\infty$, with a power law and a stretched exponential decay, respectively. 
These models have been studied in great detail (see e.g.\ \cite{godreche2003dynamics,grosskinsky2003condensation,armendariz2009thermodynamic,armendariz2013zero}), and we have $\rho_c<\infty$ when $\gamma\in (0,1)$ or $\gamma =1$ and $b>2$. If the density $\rho >\rho_c$ the system exhibits condensation, i.e.\ it phase separates into a bulk part at density $\rho_c$ and a condensate, where a finite fraction of all particles concentrates in a single site.  Accordingly, $\bar f^1$ is the stationary measure with maximal density $\rho_c$ of the birth death chain with master equation \eqref{eq:zrp_diff_f_k}. 

Intuitively, the dynamic mechanism of condensation in this model is due to the decreasing jump rates $g(k)$ leading to an effective attraction between particles on sites with a large occupation number. The system exhibits an interesting coarsening phenomenon, where over time the condensed phase concentrates on a decreasing number of lattice sites with increasing occupation numbers. While stationary measures are fully understood, there are only partial rigorous results so far on this dynamic question \cite{beltran2015martingale}, and it has been studied heuristically in \cite{godreche2003dynamics,godreche2016coarsening} and also \cite{jatuviriyapornchai2016coarsening} in terms of scaling solutions of \eqref{eq:zrp_diff_f_k}. 
While for initial conditions with $\rho =m_1 (0)\leq\rho_c$ ergodicity is expected to apply as discussed in Section \ref{sec:ergo}, for $\rho >\rho_c$ the solution to \eqref{eq:zrp_diff_f_k} is expected to phase separate into a bulk and a condensed part
\begin{equation}
\label{ansatz}  
f_{k}(t)=f^{\mathrm{bulk}}_k (t)+f^{\mathrm{cond}}_k (t)\ .
\end{equation}
For large times the bulk part of the distribution converges as
\begin{equation}
f^{\mathrm{bulk}} (t)\to \bar f^1\quad\mbox{as }t\to\infty\quad\mbox{with}\quad \sum_{k\geq 0}k\bar f_k^1 =\rho_c
\label{bulkon}
\end{equation}
in analogy with the discussion in Section \ref{sec:ergo}. The condensed part evolves indefinitely according to the scaling ansatz
\begin{equation}
\label{eq:scaling_f}
f^{\mathrm{cond}}_{k}(t)\simeq\epsilon_{t}^{2}h(u),\quad\mbox{with}\quad u=k\epsilon_{t}\quad\mbox{with scaling parameter}\quad\epsilon_{t}=t^{-\frac{1}{\gamma+1}}\to 0
\end{equation}
as $t\to\infty$. This part of the distribution concentrates on increasing occupation numbers of order $u/\epsilon_t$ describing the coarsening dynamics, and it contains a finite fraction of the mass
\begin{equation}\label{condiv}
\sum_{k\geq 1} kf^{\mathrm{cond}}_{k}(t) =\int_0^{\infty}uh(u)\, du=\rho-\rho_c\ .
\end{equation}
Plugging the ansatz into \eqref{eq:zrp_diff_f_k} leads heuristically to differential equations characterizing $h(u)$ as studied in \cite{godreche2003dynamics} for $\gamma =1$ and in \cite{godreche2016coarsening} for $\gamma<1$, which exhibit a unimodal bump corresponding to the mass distribution in the condensed phase.


With the first moment being conserved, the simplest characterization of condensation dynamics is given by the second moment of the occupation numbers, $m_2(t)$. Using the scaling ansatz \eqref{ansatz}, \eqref{eq:scaling_f} and computing $\mathcal{L}_{f(t)} k^2$, with \eqref{bulkon} and \eqref{condiv} this is dominated by the condensed part and diverges as a power law
\[
m_2(t)=\sum_{k\geq 0} k^2f_k(t) \sim \epsilon_t^{-1} =t^{\frac{1}{\gamma+1}} \quad\mbox{as }t\to\infty .
\]
Note that in a finite particle system with large but fixed $L$, the coarsening regime eventually saturates to the stationary behaviour with a single condensate site remaining. From detailed rigorous results on phase separation for canonical stationary measures $\pi_{L,N}$ in the thermodynamic limit (see e.g.\ \cite{armendariz2009thermodynamic}), we know that for $\gamma\in (0,1)$ or $\gamma =1$ and $b>3$ the second moment is dominated by the condensate and behaves as
\[
\bar m_2 :=\E \Big[\frac{1}{L}\sum_{x\in\Lambda} \eta_x^2 \Big] = L(\rho -\rho_c )^2 (1+o(1))\quad\mbox{as }L,N\to\infty\ ,\ N/L\to\rho >\rho_c \ .
\]
Equating $m_2 (t)=\bar m_2$ we can heuristically estimate the expected saturation time to scale as $L^{1+\gamma }$ with the system size $L$.

\subsection{Explosive condensation processes and gelation\label{ecp}}
Explosive condensation processes (ECP) have been introduced in \cite{waclaw2012explosive} and further studied in \cite{evans2014condensation,chau2015explosive} on a heuristic level. The jump rates are of the form
\begin{equation}\label{ECPrates}
c(k,l)={k}^{\lambda}(d+l^{\lambda})\quad\mbox{with parameters }\lambda >0\mbox{ and }d>0 ,
\end{equation}
diverging super-linearly with occupation numbers on departure and target site for $\lambda >1$. For $\lambda=1$ this model is called the inclusion process which has been studied on a rigorous level in \cite{grosskinsky2011condensation,grosskinsky2013dynamics}, and which is also covered by our result due to sublinear rates. 
ECPs are also related to aggregation models with collision kernels corresponding to $c(k,l)$, which have attracted significant research interest (see e.g.\ \cite{chau2015explosive,esenturk2017} and references therein). 

The rates \eqref{ECPrates} satisfy condition \eqref{eq:conmisrate} and we have product measures of the form \eqref{eq:marg} with $\phi_c=1$ and 
\begin{align}
w(n)&=\frac{\Gamma(d+n)}{n!\Gamma(d)}\sim n^{d-1} &\mbox{for } \lambda=1,\nonumber\\
w(n)&=\prod_{k=1}^n\frac{(k-1)^\lambda+d}{k^\lambda} \sim n^{-\lambda} &\mbox{for } \lambda\neq 1 \label{tails}
\end{align}
for all $d\geq 0$. 
Therefore, $\rho_c < \infty$ for $\lambda >2$ and again we expect $f(t)\to \bar f^{\phi}$ as $t\to \infty$ for all initial conditions with $m_1(0)=\rho \leq \rho_c$. 
If $\rho>\rho_c$, we do not expect a scaling solution as for ZRPs this time, but a behaviour corresponding to instantaneaous loss of mass in the distribution $f(t)$ so that $m_1 (t)=\rho_c <\rho$ for all $t>0$. This has not been studied in general for $d>0$, but there exists strong heuristic evidence for $d=0$ as presented below.

The exchange-driven growth model studied in \cite{ben2003exchange} corresponds to rates \eqref{ECPrates} in the degenerate case $d=0$, and provides a detailed heuristic analysis of the scaling solution for the condensed part. Note that in this case $w(n)=\delta_{0,n}$ and the mean-field equation has an absorbing state corresponding to $\bar f_k=\delta_{0,k}$ as the only stationary distribution for all $\lambda >0$, effectively setting $\rho_c =0$. Still, $m_1(t)$ is conserved and the dynamics of the particle system is not irreducible, more and more lattice sites empty over time and cannot get occupied again thereafter. 
Since $\rho_c =0$ all initial conditions with $\rho =m_1 (0)>0$ lead to phase separated solutions of the form \eqref{ansatz}, 
now with $f_k^{\mathrm{bulk}}(t)\to \delta_{k,0}$. 
The results reported in \cite{ben2003exchange} refer to $f_k^{\mathrm{cond}}$, which for $\lambda <2$ again exhibits a scaling form as $t\to\infty$,
\begin{equation}
f_{k}^{\mathrm{cond}}(t)=\epsilon_{t}^{2}h(u),\quad\mbox{with}\quad u=k\epsilon_{t}\ ,\quad \epsilon_{t}=\tau_t^{-\alpha}\quad\mbox{and}\quad \alpha=\frac{1}{2-\lambda}\ ,
\end{equation}
on a changed time scale $\tau_t=\int_0^t dt'm_{\lambda}(t')$. 
The scaling function again satisfies a second-order linear differential equation, 
and for $\lambda >2$ there is no solution to the limit dynamics \eqref{eq:mis_diff_f_k}, which exhibits instantaneous blow up of second moments -- also called gelation in the context of aggregation models (see e.g. \cite{ball1986becker}). On the level of the particle system this corresponds to the explosive condensation phenomenon studied in \cite{waclaw2012explosive,evans2014condensation,chau2015explosive} for $d>0$, where the time to reach the condensed state vanishes with increasing system size even in one-dimensional geometries. On the complete graph with $d=0$ the behaviour can again be characterized through the second moment, as reported in \cite{ben2003exchange} we have as $t\to\infty$
\begin{equation}\label{exdriven}
m_2(t) \sim  
\left\{ \begin{array}{cl}
t^{\beta} &,\ \lambda <3/2 \\
\exp(Ct)&,\ \lambda=3/2\mbox{  for some }C>0 \\
(t_c-t)^{\beta} &,\ 3/2<\lambda< 2\mbox{  for some }t_c >0 \\
\infty &,\ \lambda >2
\end{array} \right.\ .
\end{equation}
The dynamical exponent for the power law cases above is given by $\beta=(3-2\lambda)^{-1}$, and for $\lambda >3/2$ the system exhibits finite-time blow up at time $t_c$, which vanishes for $\lambda \nearrow 2$. 
For the inclusion process with $\lambda =1$ additional duality methods are available, and the above analysis can actually be made rigorous \cite{prep}, with recent results also on regular lattices instead of complete graphs \cite{carinci2017exact}.

\section{Discussion \label{sec:discussion}}

We have established the mean-field equation \eqref{eq:mis_diff_f_k} as the limit dynamics of empirical measures and single sites for stochastic particle systems, which provides an important ingredient for a rigorous approach to the coarsening dynamics of condensing stochastic particle systems. 
We have already outlined the relevance and possible future directions of research in Sections \ref{sec:properties} and \ref{sec:examples}, and we shortly discuss some restrictions and further possible generalizations below.


\begin{itemize}
	
	\item Mean-field equations \eqref{eq:mis_diff_f_k} are often used as approximations in other geometries such as symmetric or asymmetric dynamics on $d$-dimensional regular lattices. As usual, the larger the dimension the better the approximation, see e.g.\  \cite{godreche2005dynamics,chau2015explosive,evans2014condensation} for details. Since our result does not involve any time scaling, mean-field averaging of the birth and death rates \eqref{bdrates} is achieved by a diverging number of neighbours of each lattice site. This is a crucial ingredient in our proof and in fact essential for any rigorous derivation of \eqref{eq:mis_diff_f_k} without time scaling. Our arguments for empirical measures could be directly extended to graphs which are not complete but have a version of the above property.
	
	\item Condensing stochastic particle systems exhibit several time scales diverging with the system size. For ZRPs this has been studied heuristically in \cite{godreche2005dynamics}, 
	some of which have been identified recently also on a rigorous basis including hydrodynamics \cite{stamatakis2015hydrodynamic} and also metastable dynamics of the condensate \cite{beltran2012metastability,armendariz2015metastability}. 
	As we discussed in Section \ref{sec:properties} convergence in our result does not hold uniformly in time, and error estimates vanish on time scales at most of order $\log L$. This is shorter than the expected saturation time of order $L^{1+\gamma}$ explained at the end of Section \ref{zrp}, and the crossover from coarsening to stationary behaviour is not described by solutions of \eqref{eq:mis_diff_f_k}. Instead, the scaling solutions as discussed in Section \ref{sec:examples}, can be interpreted as approximations to the condensed phase on infinite regular lattices.

	\item The example of ECPs in Section \ref{ecp} with $\lambda >2$ shows that some growth conditions on the rates are necessary for convergence to \eqref{eq:mis_diff_f_k} to hold for positive times $t>0$. In case of instantaneous blow up, the single site process $\eta_x (t)$ also does not have well-defined limit dynamics for any $t>0$. The condition $\lambda <3/2$ for power law scaling without blow up \eqref{exdriven} is in fact fully compatible  with the crucial estimate \eqref{eq:limquadratic} on the quadratic variation to vanish in the limit $L\to\infty$, since this is of order $L^{2\lambda -3}$ for rates of the form \eqref{ECPrates}. This points to an interesting generalization of our main result for ECPs including also stochastic limit dynamics for $3/2<\lambda <2$ with explosion (corresponding to blow up in the deterministic equation), which is current work in progress.
	
	\item For bounded jump rates there is an alternative proof of our main result using a coupling with branching processes via the graphical construction of the process. This allows a direct approach to single site dynamics and the propagation of chaos without the law of large numbers for empirical measures, but requires a significantly stronger restriction on the rates. This proof can be found in \cite{mim2017thesis}.
	

\end{itemize}

\section*{Acknowledgements}
We are grateful to I.\ Armend\'ariz and M.\ Jara for helpful advice on the proof of the main theorem. S.\ G.\ acknowledges partial support from the Engineering and Physical Sciences Research Council (EPSRC), Grant No.\ EP/M003620/1. Part of this work has been carried out during a stay at the Institut Henri Poincare - Centre Emile Borel during the trimester 'Stochastic Dynamics Out of Equilibrium', and we are grateful for the institute's hospitality and support.


\end{document}